\documentclass[12pt]{amsart}
\usepackage[utf8]{inputenc}
\usepackage{amssymb}
\usepackage{hyperref}
\usepackage[final]{showkeys} 

\input xy
\xyoption{all}

\theoremstyle{definition}

\newtheorem{mydef}{Definition}[section]
\newtheorem{lem}[mydef]{Lemma}
\newtheorem{thm}[mydef]{Theorem}

\newtheorem{cor}[mydef]{Corollary}

\newtheorem{hypothesis}[mydef]{Hypothesis}

\newtheorem{defin}[mydef]{Definition}

\newtheorem{remark}[mydef]{Remark}

\newtheorem{fact}[mydef]{Fact}

\newcommand{\ba}{\bar{a}}

\newcommand{\Ksatpp}[2]{{#1}^{#2\text{-sat}}}
\newcommand{\Ksatp}[1]{\Ksatpp{\K}{#1}}

\newcommand{\sea}{\mathfrak{C}}

\newcommand{\ran}[1]{\text{ran}(#1)}

\newcommand{\cf}[1]{\text{cf} (#1)}
\newcommand{\seq}[1]{\langle #1 \rangle}
\newcommand{\rest}{\upharpoonright}

\newcommand{\is}{\mathfrak{i}}

\newcommand{\K}{\mathcal{K}}

\def\lea{\le}

\def\gea{\ge}

\newbox\noforkbox \newdimen\forklinewidth
\forklinewidth=0.3pt \setbox0\hbox{$\textstyle\smile$}
\setbox1\hbox to \wd0{\hfil\vrule width \forklinewidth depth-2pt
 height 10pt \hfil}
\wd1=0 cm \setbox\noforkbox\hbox{\lower 2pt\box1\lower
2pt\box0\relax}
\def\unionstick{\mathop{\copy\noforkbox}\limits}
\newcommand{\nf}{\unionstick}

\setbox0\hbox{$\textstyle\smile$}
\setbox1\hbox to \wd0{\hfil{\sl /\/}\hfil} \setbox2\hbox to
\wd0{\hfil\vrule height 10pt depth -2pt width
               \forklinewidth\hfil}
\newbox\doesforkbox
\setbox\doesforkbox\hbox{\lower 0pt\box1 \lower
2pt\box2\lower2pt\box0\relax}
\def\nunionstick{\mathop{\copy\doesforkbox}\limits}
\newcommand{\nnf}{\nunionstick}

\newcommand{\nfs}[4]{#2 \nf_{#1}^{#4} #3}
\newcommand{\nnfs}[4]{#2 \nnf_{#1}^{#4} #3}

\def\1nf{\unionstick^{(1)}}

\def\2nf{\unionstick^{(2)}}
\def\3nf{\unionstick^{(3)}}

\newcommand{\gtp}{\text{gtp}}

\newcommand{\gS}{\text{gS}}

\newcommand{\hanfe}[1]{\beth_{\left(2^{#1}\right)^+}}

\newcommand{\LS}{\text{LS}}

\title[Primes in fully good AECs]{Building prime models in fully good abstract elementary classes}
\date{\today}

\author{Sebastien Vasey}
\email{sebv@cmu.edu}
\urladdr{http://math.cmu.edu/\textasciitilde svasey/}
\address{Department of Mathematical Sciences, Carnegie Mellon University, Pittsburgh, Pennsylvania, USA}
\thanks{This material is based upon work done while the author was supported by the Swiss National Science Foundation under Grant No.\ 155136.}

\begin{document}

\begin{abstract} 

  We show how to build primes models in classes of saturated models of abstract elementary classes (AECs) having a well-behaved independence relation:

  \begin{thm}\label{abstract-thm-0}
    Let $\K$ be an almost fully good AEC that is categorical in $\LS (\K)$ and has the $\LS (\K)$-existence property for domination triples. For any $\lambda > \LS (\K)$, the class of Galois saturated models of $\K$ of size $\lambda$ has prime models over every set of the form $M \cup \{a\}$.
  \end{thm}

  This generalizes an argument of Shelah, who proved the result when $\lambda$ is a successor cardinal.
\end{abstract}

\maketitle

\tableofcontents

\section{Introduction}

Prime models (over sets) are a crucial ingredient in the proof of Morley's categoricity theorem \cite{morley-cip}. Morley's construction gives a \emph{primary model}: a model whose universe can be enumerated so that the type of each element is isolated over the previous ones. This construction can be generalized to certain non-elementary context such as homogeneous model theory \cite{sh3} and even finitary abstract elementary classes \cite{finitary-aec}.

In general abstract elementary classes (AECs), it seems that the construction breaks down due to the lack of even rudimentary compactness: it is not clear how to define a usable generalization of the notion of an isolated type. Shelah \cite[Section III.4]{shelahaecbook} works around this difficulty by assuming that the class satisfies an axiomatization of superstable forking for its models of size $\lambda$ (in Shelah's terminology, $\K$ has a successful good $\lambda$-frame) and uses \emph{domination} to build for every saturated $M$ of size $\lambda^+$ and every element $a$ a saturated model $N$ containing $M \cup \{a\}$ and prime in the class of saturated models of $\K$ size $\lambda^+$. Here, saturation is defined in terms of Galois (orbital) types.

Shelah shows \cite[Chapter II]{shelahaecbook} that the assumption of existence of a successful good $\lambda$-frame follows from strong \emph{local} hypotheses: categoricity in $\lambda$, $\lambda^+$, a medium number of models in $\lambda^{++}$, and set-theoretic hypotheses such as $2^\lambda < 2^{\lambda^+} < 2^{\lambda^{++}}$. In \cite{ss-tame-jsl, indep-aec-v6-toappear}, we showed that successful good frames can also be built assuming that the class satisfies \emph{global hypotheses}: amalgamation, categoricity in some high-enough cardinal, and a locality property called full tameness and shortness. It is known that amalgamation and the locality property both follow from categoricity and a large cardinal axiom \cite{makkaishelah, tamelc-jsl}. The global hypotheses actually enable us to build the global generalization of a successful good $\lambda$-frame: what we call an almost fully good independence relation (see Definition \ref{fully-good-def}). In this paper, we show that Shelah's construction of prime models generalizes to this global setup and $\lambda^+$ can be replaced by a limit cardinal. Thus we obtain a general construction of primes (in an appropriate class of saturated models) that works assuming only the existence of a well-behaved independence notion (this is Theorem \ref{abstract-thm-0} from the abstract).

Recently, \cite[Theorem 0.2]{ap-universal-v9} showed that assuming the global hypotheses above, existence of primes over \emph{every} set of the form $M \cup \{a\}$ implies categoricity on a tail of cardinals. Unfortunately, we cannot use the construction of prime models of this paper to deduce a new categoricity transfer in the global framework: the catch is that we only get existence of primes in the subclass of saturated models of $\K$: Given $M$ and $a$ with $M$ saturated, we obtain $N$ saturated that is prime over $M \cup \{a\}$ \emph{in the class of saturated models}. That is (roughly\footnote{the precise statement uses Galois types, see Definition \ref{prime-def}.}), if $N'$ contains $M \cup \{a\}$ \emph{and is saturated}, then there exists a $\K$-embedding $f: N \rightarrow N'$ fixing $M$ and $a$. We cannot conclude anything if $M$ is not saturated (even if we know it is $\lambda$-saturated for some $\lambda < \|M\|$). This contrasts sharply with the more uniform results from the first-order context (in a totally transcendental theory, a prime model exists over every set) and finitary AECs \cite[Lemma 5.4]{categ-finit} (in a simple $\aleph_0$-stable finitary AEC with the extension property, an $f$-primary model exists over every set).

We can however use the construction of this paper to obtain that in the global framework, categoricity on a tail of cardinals implies the existence of primes. This gives a converse to \cite[Theorem 0.2]{ap-universal-v9} (we asked if such a converse was true in \cite[Conjecture 5.22]{ap-universal-v9}). A full proof will appear elsewhere (in the final version of one of \cite{ap-universal-v9, categ-primes-v3, downward-categ-tame-v4}), but we state the result as an additional motivation:

\begin{cor}
  Let $\K$ be a fully $\LS (\K)$-tame and short AEC with amalgamation. Write $H_1 := \hanfe{\LS (\K)}$. Assume that $\K$ is categorical in \emph{some} cardinal $\lambda \ge H_1$. The following are equivalent:

  \begin{enumerate}
  \item $\K$ is categorical in \emph{all} $\lambda' \ge H_1$.
  \item $\K_{\ge H_1}$ has primes over all sets of the form $M \cup \{a\}$.
  \end{enumerate}
\end{cor}

The background required to read this paper is a basic knowledge of AECs (for example Chapters 4-12 of Baldwin's book \cite{baldwinbook09}). Some familiarity with good frames and their generalizations (in particular the beginning of \cite{indep-aec-v6-toappear}, \cite[Section 11]{indep-aec-v6-toappear}, and Shelah's construction of primes \cite[Section III.4]{shelahaecbook}) would be helpful but we state all the necessary definitions here. We rely on a few results from \cite{jarden-tameness-apal, indep-aec-v6-toappear, vv-symmetry-transfer-v3} but they are used as black boxes: little understanding of the material there is needed.

This paper was written while working on a Ph.D.\ thesis under the direction of Rami Grossberg at Carnegie Mellon University and I would like to thank Professor Grossberg for his guidance and assistance in my research in general and in this work specifically. I also thank a referee for suggestions that helped refocus the topic of the paper and improve its presentation. 

\section{Background}

We give some background on independence that will be used in the next section. We assume familiarity with the basics of AECs as laid out in e.g.\ \cite{baldwinbook09} or the forthcoming \cite{grossbergbook}. We will use the notation from the preliminaries of \cite{sv-infinitary-stability-afml}. All throughout this section, we assume:

\begin{hypothesis}\label{ap-hyp}
  $\K$ is an AEC with amalgamation.
\end{hypothesis}

This will mostly be assumed throughout the paper (Hypothesis \ref{good-hyp} implicitly implies it by Definition \ref{fully-good-def}.(\ref{ap})). Note however that assuming high-enough categoricity and a large cardinal axiom, it will hold on a tail \cite{tamelc-jsl}.

We will work use a global forking-like independence notion that has the basic properties of forking in a superstable first-order theory. This is a stronger notion than Shelah's good frame \cite[Chapter II]{shelahaecbook} because in good frames forking is only defined for types of length one. We invite the reader to consult \cite{indep-aec-v6-toappear} for more explanations and motivations on global and local independence notions.

\begin{defin}[Definition 8.1 in \cite{indep-aec-v6-toappear} and Definition A.2 in \cite{ap-universal-v9}]\label{fully-good-def} \
  $\is = (\K, \nf)$ is an \emph{almost fully good independence relation} if:

  \begin{enumerate}
  \item $\K$ is an AEC satisfying the following structural assumptions:
    \begin{enumerate}
    \item $\K_{<\LS (K)} = \emptyset$ and $\K \neq \emptyset$.
    \item\label{ap} $\K$ has amalgamation, joint embedding, and no maximal models.
    \item $\K$ is stable in all cardinals.
    \end{enumerate}
  \item \begin{enumerate}
    \item $\is$ is a $(<\infty, \ge \LS (K))$-independence relation (see \cite[Definition 3.6]{indep-aec-v6-toappear}). That is, $\nf$ is a relation on quadruples $(M, A, B, N)$ with $M \lea N$ and $A, B \subseteq |N|$ satisfying invariance, monotonicity, and normality. We write $\nfs{M}{A}{B}{N}$ instead of $\nf (M, A, B, N)$, and we also say $\gtp (\ba / B; N)$ does not fork over $M$ for $\nfs{M}{\ran{\ba}}{B}{N}$.
    \item $\is$ has base monotonicity, disjointness ($\nfs{M}{A}{B}{N}$ implies $A \cap B \subseteq |M|$), symmetry, uniqueness (whenever $M \lea N$ and $p, q \in \gS^{<\infty} (N)$ do not fork over $M$ and are such that $p \rest M = q \rest M$, then $p = q$), and the local character properties:
      \begin{enumerate}
        \item If $p \in \gS^\alpha (M)$, there exists $M_0 \lea M$ with $\|M_0\| \le |\alpha| + \LS (K)$ such that $p$ does not fork over $M_0$.
        \item If $\seq{M_i : i \le \delta}$ is increasing continuous, $p \in \gS^\alpha (M_\delta)$ and $\cf{\delta} > \alpha$, then there exists $i < \delta$ such that $p$ does not fork over $M_i$.
      \end{enumerate}
    \item $\is$ has the following weakening of the extension property: for any $M \lea N$ and any $p \in \gS^\alpha (M)$, there exists $q \in \gS^\alpha (N)$ that extends $p$ and does not fork over $M$ provided at least one of the following conditions hold:

  \begin{enumerate}
    \item $M$ is saturated.
    \item $\|M\| = \LS (\K)$.
    \item $\alpha \le \LS (\K)$.
  \end{enumerate}
    \item\label{witness-prop} $\is$ has the left and right $(\le \LS (K))$-witness properties: $\nfs{M}{A}{B}{N}$ if and only if for all $A_0 \subseteq A$ and $B_0 \subseteq B$ with $|A_0| + |B_0| \le \LS (K)$, we have that $\nfs{M}{A_0}{B_0}{N}$.
    \item\label{full-model-cont} $\is$ has full model continuity: For all limit ordinals $\delta$, if for $\ell < 4$, $\seq{M_i^\ell : i \le \delta}$ are increasing continuous such that for all $i < \delta$, $M_i^0 \lea M_i^\ell \lea M_i^3$ for $\ell = 1,2$ and $\nfs{M_i^0}{M_i^1}{M_i^2}{M_i^3}$, then $\nfs{M_\delta^0}{M_\delta^1}{M_\delta^2}{M_\delta^3}$.
  \end{enumerate}
  \end{enumerate}

  We say that $\K$ is \emph{almost fully good} if there exists $\nf$ such that $(K, \nf)$ is almost fully good\footnote{the relation $\nf$ is in fact unique \cite{bgkv-apal}.}.
\end{defin}
\begin{remark}\label{almost-fully-good-existence}
  In \cite[Theorem 15.1]{indep-aec-v6-toappear}, it was shown that AECs with amalgamation that satisfy a locality condition (full tameness and shortness) and are categorical in a high-enough cardinal are (on a tail) almost fully good. The threshold cardinals were improved in \cite[Appendix A]{ap-universal-v9}. We use the name ``almost'' fully good because we do not assume the full extension property, only the weakening above. The problem is that it is not known how to get the full extension property with the aforementioned hypotheses (see the discussion in Section 15 of \cite{indep-aec-v6-toappear}). 
\end{remark}

We will use (in the proof of Fact \ref{jarden-club}) that almost fully good AECs are tame (a locality property for types introduced by Grossberg and VanDieren \cite{tamenessone}). Recall that $\K$ is \emph{$\mu$-tame} (for $\mu \ge \LS (\K)$) if for any distinct $p, q \in \gS (M)$ there exists $M_0 \in \K_{\le \LS (\K)}$ with $M_0 \lea M$ such that $p \rest M_0 \neq q \rest M_0$. Using local character and uniqueness (see e.g.\ \cite[p.~15]{superior-aec}) we have:

\begin{fact}\label{fully-good-tame}
  If $\K$ is almost fully good, then $\K$ is $\LS (\K)$-tame.
\end{fact}

We will also make use of limit models (see \cite{gvv-toappear-v1_2} for history and motivation). We will use a global definition, where we permit the limit model and the base to have different sizes. This extra generality is used: in (\ref{clause-4}) in Lemma \ref{canonical-prime}, $M_i^\ell$ and $M_{i + 1}^\ell$ may have different sizes.

\begin{defin}\label{limit-def}
  Let $M_0 \lea M$ be models in $\K_{\ge \LS (\K)}$. $M$ is \emph{limit over $M_0$} if there exists a limit ordinal $\delta$ and a strictly increasing continuous sequence $\seq{N_i : i \le \delta}$ such that $N_0 = M_0$, $N_\delta = M$, and for all $i < \delta$, $N_{i + 1}$ is universal over $N_i$.

  We say that $M$ is \emph{limit} if it is limit over some $M' \lea M$.
\end{defin}

We will use the following notation to describe classes of saturated models:

\begin{defin}
  For $\lambda > \LS (\K)$, $\Ksatp{\lambda}$ is the class of $\lambda$-saturated (according to Galois types) models in $\K_{\ge \lambda}$. We order $\Ksatp{\lambda}$ with the strong substructure relation induced from $\K$.
\end{defin}

In an almost fully good AEC, classes of $\lambda$-saturated models are well-behaved and limit models are saturated. This is a combination of results of Shelah \cite[Chapter II]{shelahaecbook} and VanDieren \cite{vandieren-chainsat-apal}, and  is key in the construction of prime models of the next section.

\begin{fact}\label{uq-lim-fact}
  Assume that $\K$ is almost fully good. Then:

  \begin{enumerate}
  \item\label{uq-lim} If $M, N \in \K$ are limit and $\|M\| = \|N\|$, then $M \cong N$. In particular (if $\|M\| > \LS (\K)$), $M$ and $N$ are saturated.
  \item\label{chainsat} For any $\lambda > \LS (\K)$, $\Ksatp{\lambda}$ is an AEC with $\LS (\Ksatp{\lambda}) = \lambda$. Moreover, $\Ksatp{\lambda}$ is almost fully good.
  \end{enumerate}
\end{fact}
\begin{proof}
  Let $\is$ be an almost fully good independence relation on $\K$. 
  \begin{enumerate}
  \item Let $M, N \in \K$ be limit models with $\lambda := \|M\| = \|N\|$. First, by a back and forth argument we can assume that $M$ is limit over some $M_0 \in \K_\lambda$ and $N$ is limit over some $N_0 \in \K_\lambda$ (see \cite[Proposition 3.1]{vv-symmetry-transfer-v3}). Next, note that the restriction of $\is$ to types of length one and models of size $\lambda$ induces a good $\lambda$-frame (see \cite[Chapter II]{shelahaecbook}). The result now follows from \cite[Lemma II.4.8]{shelahaecbook} (see \cite[Theorem 9.2]{ext-frame-jml} for a detailed proof).
  \item We prove the result when $\lambda$ is a successor cardinal, and the result for $\lambda$ limit easily follows (the moreover part is also straightforward to check). So assume that $\lambda = \mu^+$. Note that $\K$ is $\mu$-superstable in the sense of \cite[Definition 5]{vandieren-chainsat-apal} (because nonforking implies nonsplitting, see the proof of \cite[Fact 4.8.(2)]{vv-symmetry-transfer-v3}). Similarly, $\K$ is $\mu^+$-superstable. By the first part, limit models of cardinality $\mu^+$ are unique. Therefore we can apply \cite[Theorem 22]{vandieren-chainsat-apal} which tells us that the union of an increasing chain of $\mu^+$-saturated models is $\mu^+$-saturated. That $\LS (\Ksatp{\mu^+}) = \mu^+$ follows from stability and the other axioms of an AEC are straightforward to check.
  \end{enumerate}
\end{proof}
\begin{remark}
  Fact \ref{uq-lim-fact}.(\ref{chainsat}) is an improvement on the threshold cardinal in \cite{bv-sat-v3} (where it is shown that that $\Ksatp{\lambda}$ is an AEC for all $\lambda \ge \hanfe{\LS (\K)}$).
\end{remark}

Domination will be the notion replacing isolation in this paper's construction of prime models:

\begin{defin}\label{dom-def}
  Let $\is = (\K, \nf)$ be an almost fully good independence relation. $(a, M, N)$ is a \emph{domination triple} if $M \lea N$, $a \in |N| \backslash |M|$, and for any $N' \gea N$ and any $B \subseteq |N'|$, if $\nfs{M}{a}{B}{N'}$, then $\nfs{M}{N}{B}{N'}$.
\end{defin}
\begin{remark}\label{dom-uq-triple}
  This is a variation on Shelah's uniqueness triples \cite[Definition II.5.3]{shelahaecbook}. In fact by \cite[Lemma 11.7]{indep-aec-v6-toappear}, uniqueness triples and domination triples coincide in our framework (this will be used in the proof of Fact \ref{jarden-club}, although an understanding of the exact definition of uniqueness triples is not needed for this paper).
\end{remark}

A key property is the existence property for domination triples\footnote{This is analogous to Shelah's definition of a \emph{weakly successful} good $\lambda$-frame \cite[Definition III.1.1]{shelahaecbook} which means the frame has the existence property for uniqueness triples.}:

\begin{defin}\label{exist-triple-def}
  Let $\is = (\K, \nf)$ be an almost fully good independence relation and let $\lambda \ge \LS (\K)$. We say that \emph{$\is$ has the $\lambda$-existence property for domination triples} if for every $M \in \K_\lambda$ and every nonalgebraic $p \in \gS (M)$, there exists a domination triple $(a, M, N)$ so that $p = \gtp (a / M; N)$.
\end{defin}

The existence property for domination triples is a reasonable hypothesis: if the independence relation does not have it, we can restrict ourselves to a subclass of saturated models (see the moreover part of Fact \ref{uq-lim-fact}.(\ref{chainsat})).

\begin{fact}[Lemma 11.12 in \cite{indep-aec-v6-toappear}]\label{existence-dom-triples}
  Let $\is = (\K, \nf)$ be an almost fully good independence relation. For every $\lambda > \LS (\K)$, $\is \rest \Ksatp{\lambda}$ (the restriction of $\is$ to $\lambda$-saturated models) has the $\lambda$-existence property for domination triples.
\end{fact}

Finally, we recall the definition of prime models in the framework of abstract elementary classes. This does not need amalgamation and is due to Shelah \cite[Section III.3]{shelahaecbook}. While it is possible to define what it means for a model to be prime over an arbitrary set (see \cite[Definition 5.1]{ap-universal-v9}), here we focus on primes over sets of the form $M \cup \{a\}$. The technical point in the definition is that since we are not working inside a monster model, how $M \cup \{a\}$ is embedded matters. Thus we use a formulation in terms of Galois types: instead of saying that $N$ is prime over $M \cup \{a\}$, we say that $(a, M, N)$ is a \emph{prime triple}:

\begin{defin}\label{prime-def}
  Let $\K$ be an AEC (not necessarily with amalgamation). 

  \begin{enumerate}
    \item A \emph{prime triple} is $(a, M, N)$ such that $M \lea N$, $a \in |N| \backslash |M|$ and for every $N' \in \K$, $a' \in |N'|$ such that $\gtp (a / M; N) = \gtp (a' / M; N')$, there exists $f: N \xrightarrow[M]{} N'$ so that $f (a) = a'$.
    \item We say that $\K$ \emph{has primes} if for $M \in \K$ and every nonalgebraic $p \in \gS (M)$, there exists a prime triple representing $p$, i.e.\ there exists a prime triple $(a, M, N)$ so that $p = \gtp (a / M; N)$.
    \item We define localizations such as ``$\K_\lambda$ has primes'' or ``$\Ksatp{\lambda}_\lambda$ has primes'' in the natural way (in the second case, we ask that all models in the definition be saturated).
  \end{enumerate}
\end{defin}

\section{Building primes over saturated models}

We show that in almost fully good AECs, there exists primes among the saturated models (see Definition \ref{prime-def}). For models of successor size, this is shown in \cite[Claim III.4.9]{shelahaecbook} (or in \cite{jarden-prime} with slightly weaker hypotheses). We generalizes Shelah's proof to limit sizes here. This is the core of the paper. Throughout this section, we assume:

\begin{hypothesis}\label{good-hyp} \
  \begin{enumerate}
  \item $\K$ is an almost fully good AEC, as witnessed by $\is = (\K, \nf)$.
  \item $\K$ is categorical in $\LS (\K)$.
  \item\label{existence-prop-hyp} $\is$ has the $\LS (\K)$-existence property for domination triples (see Definition \ref{exist-triple-def}).
  \end{enumerate}
\end{hypothesis}

We consider theses hypotheses reasonable: Remark \ref{almost-fully-good-existence} gives conditions under which an AEC is almost fully good and Fact \ref{existence-dom-triples} shows that we can then restrict it to a subclass of saturated models to obtain the existence property for domination triples and categoricity in $\LS (\K)$.

Note that Hypothesis \ref{good-hyp}.(\ref{existence-prop-hyp}) is used in the proof of Fact \ref{jarden-club}. We do not know whether it follows from the other two hypotheses.

We start by showing that domination triples are closed under unions. This is a key consequence of full model continuity.

\begin{lem}\label{uq-cont}
  Let $\seq{M_i : i < \delta}, \seq{N_i : i < \delta}$ be increasing and assume that $(a, M_i, N_i)$ are domination triples for all $i < \delta$. Then $(a, \bigcup_{i < \delta} M_i, \bigcup_{i < \delta} N_i)$ is a domination triple.
\end{lem}
\begin{proof}
  For ease of notation, we work inside a monster model $\sea$ and write $A \nf_M B$ for $\nfs{M}{A}{B}{\sea}$. Let $M_\delta := \bigcup_{i < \delta} M_i$, $N_\delta := \bigcup_{i < \delta} N_i$. Assume that $a \nf_{M_\delta} N$ with $M_\delta \lea N$ (by extension for types of length one, we can assume this without loss of generality). By local character, for all sufficiently large $i < \delta$, $a \nf_{M_i} N$. By definition of domination triples, $N_i \nf_{M_i} N$. By full model continuity, $N_\delta \nf_{M_\delta} N$.
\end{proof}

The conclusion of the next fact is a key step in Shelah's construction of a successor frame in \cite[Chapter II]{shelahaecbook}. The fact says that if $M^0 \lea M^1$ are of the same successor size, then their resolutions satisfy a natural independence property on a club. In the framework of this paper, this is due to Jarden \cite{jarden-tameness-apal}. To give the reader a feeling for the difficulties encountered, we first explain in the proof how the (straightforward) first-order argument fails to generalize.

\begin{fact}\label{jarden-club}
  For every $\mu \ge \LS (\K)$, for every $M^0 \lea M^1$ both in $\K_{\mu^+}$, if $\seq{M_i^\ell : i < \mu^+}$ are increasing continuous resolutions of $M^\ell$ and all are limit models\footnote{And hence if $\mu > \LS (\K)$ are saturated (Fact \ref{uq-lim-fact}.(\ref{uq-lim})).} in $\K_\mu$, $\ell = 0, 1$, then the set of $i < \mu^+$ so that $\nfs{M_i^0}{M^0}{M_i^1}{M^1}$ is a club.
\end{fact}
\begin{proof}
  Let us first see how the first-order argument would go. By local character, for every $i < \mu^+$, there exists $j_i < \mu^+$ such that $\nfs{M_{j_i}^0}{M_i^1}{M^0}{M^1}$. Pick $i^\ast < \mu^+$ such that $j_i < i^\ast$ for every $i < i^\ast$. Using symmetry and the fact that forking is witnessed by a formula (this is where we use the first-order theory), it is then straightforward to see that $\nfs{M_{i^\ast}^0}{M^0}{M_{i^\ast}^1}{M^1}$. Thus $i^\ast$ has the desired property, and the argument shows we can find a closed unbounded subset of such $i^\ast$. Here however we do not have that forking is witnessed by a formula, or even a finite set (we only have the $\LS (\K)$-witness property, see Definition \ref{fully-good-def}.(\ref{witness-prop})).

  Full model continuity (Definition \ref{fully-good-def}.(\ref{full-model-cont})) seems to be the replacement we are looking for, but in the argument above we do \emph{not} have that $M_{j_i}^0 \lea M_i^1$ so cannot use it! It is open whether the appropriate generalization of full model continuity holds here.
  
  On to the actual proof. We show the result when $\mu = \LS (\K)$. Once this is done, if $\mu > \LS (\K)$ we can apply the ``$\mu = \LS (\K)$'' case to the class $\Ksatp{\mu}$ (by Fact \ref{uq-lim-fact}.(\ref{chainsat}) it is an almost fully good AEC and $\LS (\Ksatp{\mu}) = \mu$).

  We now want to apply \cite[Theorem 7.8]{jarden-tameness-apal}. The conclusion there is that for any model $M^0, M^1 \in \K_{\mu^+}$, $M^0 \lea_{\mu^+}^{\text{NF}} M^1$ if and only if $M^0 \lea M^1$, where $M^0 \lea_{\mu^+}^\text{NF} M^1$ is defined to hold if and only if there exists increasing continuous resolutions of $M^0$ and $M^1$ as here. Let us check that the hypotheses of \cite[Theorem 7.8]{jarden-tameness-apal} are satisfied. First, amalgamation in $\LS (\K)^+$ and $\LS (\K)$-tameness hold (by definition of an almost fully good AEC and Fact \ref{fully-good-tame}). Second, \cite[Hypothesis 6.5]{jarden-tameness-apal} holds: $\K$ is categorical in $\LS (\K)$, has a semi-good $\LS (\K)$-frame (this is weaker than the existence of an almost fully good independence relation, in fact the frame will be good), satisfies the conjugation property (by \cite[III.1.21]{shelahaecbook} which tells us that conjugation holds in any good $\LS (\K)$-frame categorical in $\LS (\K)$), and has the existence property for uniqueness triples by Hypothesis \ref{good-hyp}.(\ref{existence-prop-hyp}) and Remark \ref{dom-uq-triple}. Therefore the hypotheses of Jarden's theorem are satisfied so its conclusion holds.
\end{proof}

We can now generalize the proof of \cite[Claim III.4.3]{shelahaecbook} to limit cardinals. Roughly, it tells us that every nonalgebraic type over a saturated model has a resolution into domination triples.

\begin{lem}\label{canonical-prime}
  Let $\lambda > \LS (\K)$ and let $\delta := \cf{\lambda}$. Let $M^0 \in \K_\lambda$ be saturated and let $p \in \gS (M^0)$ be nonalgebraic. Then there exists a saturated $M^1 \in \K_\lambda$, an element $a \in |M^1|$, and increasing continuous resolutions $\seq{M_i^\ell : i \le \delta}$ of $M^\ell$, $\ell = 0,1$ such that for all $i < \delta$:

  \begin{enumerate}
    \item\label{clause-1} $p = \gtp (a / M^0; M^1)$.
    \item\label{clause-2} $a \in |M_0^1|$.
    \item\label{clause-3} $p$ does not fork over $M_0^0$.
    \item\label{clause-4} For $\ell = 0,1$, $M_i^\ell \in \K_{[\LS (\K), \lambda)}$ and $M_{i + 1}^\ell$ is limit over $M_i^\ell$.
    \item\label{clause-5} $(a, M_i^0, M_i^1)$ is a domination triple.
  \end{enumerate}
\end{lem}
\begin{proof}
  For $\ell = 0,1$, we first choose by induction $\seq{N_i^\ell : i \le \lambda}$ increasing continuous and an element $a$ that will satisfy some weaker requirements. In the end, we will rename the $N_i^\ell$'s to get the desired $M_i^\ell$'s and $M^1$. We require that for all $i < \lambda$:

  \renewcommand{\theenumi}{\roman{enumi}}%
  \begin{enumerate}
    \item $N_0^0 \lea M^0$ and $p$ does not fork over $N_0^0$.
    \item $a \in |N_0^1|$ and $p \rest N_0^0 = \gtp (a / N_0^0; N_0^1)$.
    \item For $\ell = 0,1$, $N_i^\ell \in \K_{|i| + \LS (\K)}$ and $N_i^0 \lea N_i^1$.
    \item $\gtp (a / N_i^0; N_i^1)$ does not fork over $N_0^0$.
    \item If $i$ is odd, and $\ell = 0,1$, then $N_{i + 1}^\ell$ is limit over $N_i^\ell$.
    \item If $i$ is even and $(a, N_i^0, N_i^1)$ is \emph{not} a domination triple, then $\nnfs{N_i^0}{N_{i}^1}{N_{i + 1}^0}{N_{i + 1}^1}$.
  \end{enumerate}

  \paragraph{\underline{This is possible}}

  First pick $N_0^0 \in \K_{\LS (\K)}$ such that $N _0^0 \lea M^0$ and $p$ does not fork over $N_0^0$. This is possible by local character. Now pick $N_0^1 \in \K_{\LS (\K)}$ such that $N_0^0 \lea N_0^1$ and there is $a \in |N_0^1|$ with $\gtp (a / N_0^0; N_0^1) = p \rest N_0^0$. This takes care of the case $i = 0$. For $i$ limit, take unions. Now assume that $i = j + 1$ is a successor. We consider several cases: \begin{itemize}
  \item If $j$ is even and $(a, N_j^0, N_j^1)$ is \emph{not} a domination triple, then there must exist witnesses $N_{j + 1}^0, N_{j + 1}^1 \in \K_{\LS (\K) + |j|}$ such that $N_j^0 \lea N_{j + 1}^0$, $N_{j + 1}^0 \lea N_{j + 1}^1$, $N_{j}^1 \lea N_{j + 1}^1$, $\nfs{N_j^0}{a}{N_{j + 1}^0}{N_{j + 1}^1}$ but $\nnfs{N_j^0}{N_j^1}{N_{j + 1}^0}{N_{j + 1}^1}$. This satisfies all the conditions (we know that $\gtp (a / N_j^0; N_j^1)$ does not fork over $N_0^0$, so by transitivity also $\gtp (a / N_{j + 1}^0; N_{j + 1}^1)$ does not fork over $N_0^0$).
  \item If $j$ is even and $(a, N_j^0, N_j^1)$ \emph{is} a domination triple, take $N_{j + 1}^\ell := N_j^\ell$, for $\ell = 0,1$.
  \item If $j$ is odd, pick $N_i^0 \in \K_{\LS (\K) + |j|}$ limit over $N_j^0$ and $N_i^1$ limit over $N_i^0$ and $N_j^1$ so that $\gtp (a / N_i^0; N_i^1)$ does not fork over $N_0^0$. This is possible by the extension property for types of length one.
  \end{itemize}

  \paragraph{\underline{This is enough}}

  By the odd stages of the construction, and basic properties of universality, for all $i < \lambda$, $\ell = 0,1$, $N_{i + 2}^\ell$ is universal over $N_i^\ell$. Thus for $\ell = 0,1$ and $i \le \lambda$ a limit ordinal, $N_i^\ell$ is limit. In particular, by Fact \ref{uq-lim-fact}.(\ref{uq-lim}), $N_\lambda^\ell$ is saturated. By uniqueness of saturated models, $N_\lambda^0 \cong_{N_0^0} M^0$. By uniqueness of the nonforking extension, without loss of generality $N_\lambda^0 = M^0$. Now let $C$ be the set of \emph{limit} $i < \lambda$ such that $(a, N_i^0, N_i^1)$ is a domination triple. We claim that $C$ is a club:

  \begin{itemize}
    \item $C$ is closed by Lemma \ref{uq-cont}.
    \item $C$ is unbounded: given $\alpha < \lambda$, let $\mu := |\alpha| + \LS (\K)$. Let $E_\mu$ be the set of $i < \mu^+$ such that $i$ is limit and $\nfs{N_i^0}{N_{\mu^+}^0}{N_i^1}{N_{\mu^+}^1}$. By Fact \ref{jarden-club}, $E_\mu$ is a club. The even stages of the construction imply that for $i \in E_\mu$, $(a, N_i^0, N_i^1)$ is a domination triple. In other words, $E_\mu \subseteq C$. Now pick $\beta \in E_\mu \backslash (\alpha + 1)$. We have that $\alpha < \beta$ and $\beta \in E_\mu \subseteq C$. This completes the proof that $C$ is unbounded.
  \end{itemize}
      
  Let $\seq{\alpha_i : i < \delta}$ (recall that $\delta = \cf{\lambda}$) be a cofinal strictly increasing continuous sequence of elements of $C$. For $i < \delta$, $\ell = 0,1$, let $M_i^\ell := N_{\alpha_i}^\ell$ and let $M^1 := M_\lambda^1$ (note that $M^1$ is saturated by what has been observed above). This works: Clauses (\ref{clause-1}), (\ref{clause-2}), (\ref{clause-3}) are straightforward to check using the monotonicity and uniqueness properties of forking. Clause (\ref{clause-5}) holds by definition of $C$. As for (\ref{clause-4}), we have observed above that for $\ell = 0,1$, for all $i < \lambda$, $N_{i + 2}^\ell$ is universal over $N_i^\ell$. Hence for all limit ordinals $i < j < \lambda$, $N_j^\ell$ is limit over $N_i^\ell$. In particular because $C$ contains only limit ordinals, for all $i < \delta$, $N_{\alpha_{i + 1}}^\ell$ is limit over $N_{\alpha_i}^\ell$, as desired.
\end{proof}

In \cite[Claim III.4.9]{shelahaecbook}, Shelah observes that triples as in the conclusion of Lemma \ref{canonical-prime} are prime triples. For the convenience of the reader, we include the proof here. We will use the following fact which follows from the uniqueness property of forking and some renaming.

\begin{fact}[Lemma 12.6 in \cite{indep-aec-v6-toappear}]\label{uq-nf}
  For $\ell < 2$, $i < 4$, let $M_i^\ell \in \K$ be such that for $i = 1,2$, $M_0^\ell \lea M_i^\ell \lea M_3^\ell$.

  If $\nfs{M_0^\ell}{M_1^\ell}{M_2^\ell}{M_3^\ell}$ for $\ell < 2$, $f_i : M_i^1 \cong M_i^2$ for $i = 0, 1, 2$, and $f_0 \subseteq f_1$, $f_0 \subseteq f_2$, then $f_1 \cup f_2$ can be extended to $f_3 : M_3^1 \rightarrow M_4^2$, for some $M_4^2$ with $M_3^2 \lea M_4^2$.  
\end{fact}

We can now give a proof of Theorem \ref{abstract-thm-0} from the abstract. For the convenience of the reader we restate Hypothesis \ref{good-hyp} here.

\begin{thm}\label{prime-sat}
  Let $\K$ be an almost fully good AEC that is categorical in $\LS (\K)$ and has the $\LS (\K)$-existence property for domination triples.
  
  For any $\lambda > \LS (\K)$, $\Ksatp{\lambda}_{\lambda}$ has primes (see Definition \ref{prime-def}). That is, for any saturated $M \in \K_\lambda$ and any nonalgebraic $p \in \gS (M)$, there exists a triple $(a, M, N)$ such that $M \lea N$, $N \in \K_\lambda$ is saturated, $p = \gtp (a / M; N)$, and whenever $p = \gtp (b / M; N')$ with $N' \in \K_\lambda$ saturated, there exists $f: N \xrightarrow[M]{} N'$ such that $f (a) = b$.
\end{thm}
\begin{proof}
  Let $M \in \K_\lambda$ be saturated and let $p \in \gS (M)$ be nonalgebraic. We must find a triple $(a, M, N)$ such that $M \lea N$, $N \in \K_\lambda$ is saturated, $p = \gtp (a / M; N)$, and $(a, M, N)$ is a prime triple among the saturated models of size $\lambda$.

  Set $M^0 := M$ and let $\delta := \cf{\lambda}$. Let $M^1$, $a$, $\seq{M_i^\ell : i \le \delta}$ be as described by the statement of Lemma \ref{canonical-prime}. Recall (this is key) that $\|M_i^\ell\| < \lambda$ for any $i < \delta$. We show that $(a, M^0, M^1)$ is as desired. By assumption, $M^0 \lea M^1$, $p = \gtp (a / M^0; M^1)$, and $M^1 \in \K_\lambda$ is saturated. It remains to show that $(a, M^0, M^1)$ is a prime triple in $\Ksatp{\lambda}_\lambda$. Let $M' \in \Ksatp{\lambda}_\lambda$, $a' \in |M'|$ be given such that $\gtp (a' / M^0; M') = \gtp (a / M^0; M^1)$. We want to build $f: M^1 \xrightarrow[M^0]{} M'$ so that $f (a) = a'$.

  We build by induction an increasing continuous chain of embeddings $\seq{f_i : i \le \delta}$ so that for all $i \le \delta$:

  \begin{enumerate}
    \item $f_i : M_i^1 \xrightarrow[M_i^0]{} M'$.
    \item $f_i (a) = a'$.
  \end{enumerate}

  This is enough since then $f := f_\delta$ is as required. This is possible: for $i = 0$, we use that $M'$ is saturated, hence realizes $p \rest M_0^0$, so there exists $f_0: M_0^1 \xrightarrow[M_0^0]{} M'$ witnessing it, i.e.\ $f_0 (a) = a'$. At limits, we take unions. For $i = j + 1$ successor, let $\mu := \|M_j^1\| + \|M_i^0\|$. Pick $N_j \lea M'$ with $N_j \in \K_\mu$ and $N_j$ containing both $f_j[M_j^1]$ and $M_i^0$.

  By assumption, $p$ does not fork over $M_0^0$ and by assumption $p = \gtp (a' / M^0; M')$, so by monotonicity of forking, $\nfs{M_j^0}{a'}{M_{j + 1}^0}{N_j}$. We know that $(a, M_j^0, M_j^1)$ is a domination triple, hence applying $f_j$ and using invariance, $(a', M_j^0, f_j[M_j^1])$ is a domination triple. Therefore $\nfs{M_j^0}{f_j[M_j^1]}{M_{j + 1}^0}{N_j}$. By a similar argument, we also have $\nfs{M_j^0}{M_j^1}{M_{j + 1}^0}{M_{j + 1}^1}$. By Fact \ref{uq-nf}, the map $f_j \cup \text{id}_{M_{j + 1}^0}$ can be extended to a $\K$-embedding $g: M_{j + 1}^\alpha \rightarrow N_j'$ for some $N_j' \gea N_j$ of size $\mu$. Since $\mu < \lambda$ and $M'$ is saturated, there exists $h: N_j' \xrightarrow[N_j]{} M'$. Let $f_{j + 1} := h \circ g$.
\end{proof}

\bibliographystyle{amsalpha}
\bibliography{primes-over-saturated}

\end{document}